\documentclass{article}

\usepackage[applemac]{inputenc}		

\usepackage{authblk}

\usepackage{amsmath}
\usepackage{amsthm}
\usepackage{amsxtra}
\usepackage{pxfonts}
\usepackage{mathrsfs}
\usepackage{marvosym}

\usepackage{braket}
\usepackage{tikz} 
\usetikzlibrary{matrix,arrows} 
\usepackage{hyperref}

\newcommand{\R}{\mathbb R}
\newcommand{\C}{\mathbb C}

\renewcommand{\mod}[1]{\lvert #1 \rvert}
\newcommand{\norm}[1]{\left\| #1 \right\|}

\newcommand{\I}{\text{i}}

\newcommand{\dom}[1]{\mathcal D\left(#1\right)}

\renewcommand{\H}{\mathcal H}

\newcommand{\A}{\mathcal A}
\newcommand{\V}{\mathcal V}
\renewcommand{\S}{\mathcal S}
\renewcommand{\L}{\mathcal L}
\newcommand{\LHV}{\mathcal L^\dagger(\mathcal V)}

\DeclareMathOperator{\vspan}{Span}

\DeclareMathOperator{\gr}{\mathcal G}

\theoremstyle{definition}
\newtheorem{defn}{Definition}[section]

\theoremstyle{plain}
\newtheorem{prop}[defn]{Proposition}
\newtheorem{lem}[defn]{Lemma}
\newtheorem{thm}[defn]{Theorem}
\newtheorem{coro}[defn]{Corollary}

\theoremstyle{remark}
\newtheorem{ej}[defn]{Example}
\newtheorem{rk}[defn]{Remark}

\title{Integrable representations  of involutive algebras and Ore localization
}
\author{Rodrigo Vargas Le-Bert\footnote{
\Letter {\tt rvargas@inst-mat.utalca.cl}. Supported by Fondecyt Postdoctoral Grant N\textordmasculine 3110045. 
} }
\affil{
Instituto de Matemática y Física, Universidad de Talca \\
Casilla 747, Talca, Chile
}
\date{April 5, 2011}

\begin{document}
\maketitle

\begin{abstract}
Let $\A$ be a unital algebra equipped with an involution $(\cdot)^\dagger$, and suppose  that the multiplicative set $\S\subseteq \A$ generated by the elements of the form $1+a^{\dagger} a$ contains only regular elements and satisfies  the Ore condition. We prove that:
\begin{itemize}
	\item Ultracyclic representations of $\A$ admit an integrable extension (acting on a possibly \emph{larger} Hilbert space).
	\item Integrable representations of $\A$ are in bijection with representations of the Ore localization $\A\mathcal S^{-1}$ (which we prove to be an involutive algebra).
\end{itemize}
This second result can be understood as a restricted converse to a theorem by Inoue asserting that representations of symmetric involutive algebras are integrable. 
\medskip \\
{\bf 2010 MSC}: 16S (primary); 46L (secondary).
\end{abstract}


\section{Introduction}

Unbounded operator  algebras appear in several important domains, such as quantum field theory, representations of Lie algebras and quantum groups. Consequently, there has long been an interest in developing their theory, which despite that has grown rich in technical details and relatively poor in applications (see \cite{m:Schm90, mp:DubiHenn90, m:AntoInouTrap02, m:Frag05} for complete expositions,~\cite{mp:Baga07} for a physical applications survey, and~\cite{m:MallHara05} for applications in other fields). 
Among the causes for this fact lie the inherent difficulties in the representation theory of general involutive algebras, which we summarize as follows.

Let~$\A$ be a unital algebra equipped with an involution $(\cdot)^{\dagger}$ and $\H$ a Hilbert space.
\begin{enumerate}
	\item Speaking of a representation presupposes that sum, product and involution are well-defined; however, unbounded operators are not defined all over $\H$ and cannot be blindly added or composed. Hence, one usually postulates the existence of an invariant domain $\V\subseteq\H$ for the operators which will represent $\A$, thus allowing for a pointwise definition of their sum and product. Letting
\[
\LHV = \Set{ a:\V\rightarrow\V | \V\subseteq\dom{a^*} \text{ and } a^*\V\subseteq \V },
\] 
one obtains an algebra with involution $a^\dagger = a^*|_{\V}$, and then can define a representation to be a morphism $\pi:\A\rightarrow \LHV$.
	\item Operators in $\L^\dagger(\V)$ are closable, for their adjoint is densely defined. Therefore, the natural algebraic operations with them are the so-called \emph{strong sum} and \emph{strong product,} see Definition \ref{strong_ops}. The problem arises from this, together with the fact that each $a\in\LHV$ may admit more than one closed extension. As a consequence, pointwise operations become ambiguous in a sense.
	They are too weak a version of the strong ones. For a representation $\pi$, this means that
	\[
	\bar\pi(a+b)\subseteq \bar\pi(a)+\bar\pi(b),\quad \bar\pi(ab) \subseteq \bar\pi(a)\bar\pi(b),\quad \bar\pi(a^\dagger)\subseteq \bar\pi(a)^*,
	\]
and equality does not hold in general.
\end{enumerate}

Simply put, $\LHV$ is not an object that truly captures the algebraic structure of closed operators,
thus allowing, even in the simplest cases, for the appearance of ill-behaved representations. The following two examples are archetypical of good and bad behaviour.

\begin{ej}
Let $G$ be a Lie group with Lie algebra $\mathfrak g$. Representations of the universal enveloping algebra $\A = \mathcal U(\mathfrak g)$ that arise by differentiation from unitary representations of~$G$ are called \emph{integrable} (a complete exposition is found in \cite{m:Schm90}). They are characterized by the fact that they respect the involution: 
\[
\bar\pi(a^\dagger) = \bar\pi(a)^*,\quad \forall a\in\A,
\]
a property which is taken as the definition of integrability in the general case.
Integrability is also very important when $\A$ is the observable algebra of a quantum system. 
\end{ej}
\begin{ej}
Let $M$ be any properly infinite von Neumann algebra acting on a separable Hilbert space. Then, there exists a representation $\pi$ of $\A=\C[x,y]$, the algebra of polynomials in two commuting hermitian variables,  such that $\pi(x)$ and $\pi(y)$ are essentially self-adjoint and their spectral projections generate~$M$. 
Such representations are not integrable. For more details, see \cite[Section 5]{m:Powe71} and  \cite[Section 9.4]{m:Schm90}. 
\end{ej}

%

The second  example shows that  conditions must be imposed on representations, in order to exclude pathological behaviour and make them useful in practice.
Now, during the last decade, two general approaches to the well-behaved representations have been developed.
In what follows we provide a brief description of both (a complete discussion is found in~\cite{m:AntoInouTrap02}). After that, we will be in position to comment on one important limitation that they share, and on how it is overcome in an approach that we propose in this paper.

The first one, of an algebraic flavor, is due to Schmüdgen~\cite{m:Schm02}. It generalizes the notion of integrable representations of universal enveloping algebras $\mathcal U(\mathfrak g)$, in the following way: integrable representations of $\mathcal U(\mathfrak g)$ determine unitary representations of the corresponding simply connected Lie group $G$. Those, in turn, are equivalent to representations of the Banach algebra $L^1(G)$. Now, it turns out that an integrable representation of $\mathcal U(\mathfrak g)$ can be recovered from the corresponding representation of $C_0^\infty(G)\subseteq L^1(G)$ by making use of the natural action of $\mathcal U(\mathfrak g)$ on $C_0^\infty(G)$. The generalization to arbitrary involutive algebras goes as follows: say that the involutive algebra $\A$ and the normed involutive algebra $\A_0$ are \emph{compatible} if there is an action of $\A$ on $\A_0$ such that
\[
(a\cdot x)^\dagger y = x^\dagger (a^\dagger \cdot y),\quad \forall a\in\A,\ \forall x,y\in\A_0.
\]
Then, non-degenerate, continuous representations $\pi_0$ of $\A_0$ determine what we define to be the well-behaved representations $\pi$ of $\A$ by
\[
\pi(a)\pi_0(x) = \pi_0(a\cdot x),\quad a\in\A,\ x\in\A_0.
\]

The second approach, of an analytic flavor, was proposed by Bhatt, Inoue and Ogi~\cite{m:BhatInouOgi01}. It is based on unbounded C*-seminorms. Given an involutive algebra $\A$, an unbounded C*-seminorm is a C*-seminorm $p:\A_0\rightarrow \R$, where $\A_0$ is a subalgebra of $\A$. The kernel $\mathcal I$ of such a seminorm is a bilateral ideal of $\A_0$, and the completion with respect to $p$ of $\A_0/\mathcal I$ is a C*-algebra which we denote by $A$. Now, let
\[
\mathcal N = \set{ x\in\A_0 | ax\in\A_0,\forall a\in\A},
\]
which is a left ideal of $\A$. Each representation $\pi_0:A\rightarrow B(\H)$ induces a representation $\pi:\A\rightarrow \LHV$, where
\[
\V = \vspan \Set{ \pi_0(x+\mathcal I)\xi | x\in\mathcal N, \xi\in\H },
\]
by the simple formula $\pi(a)\pi_0(x+\mathcal I)\xi = \pi_0(ax+\mathcal I)\xi$. 
Observe, however, that $\V$ might not be dense in $\H$ and, when this is the case, we say that the representation is well-behaved. One sees that  well-behaved representations satisfy $\norm{\bar\pi(a)} = p(a)$, for all $a\in\A_0$ (in general, $\norm{\bar\pi(a)}\leq p(a)$).

This second approach is more general than the first one, and can be further generalized to the case of partial involutive algebras~\cite{m:Trap06}. The precise relationship between the two approaches has been worked out in~\cite{m:Inou04}. The limitation referred to above is more apparent in the first one: obtaining well-behaved representations of $\A=\mathcal U(\mathfrak g)$ requires knowledge of representations of $\A_0 = C_0^\infty(G)$. If one is to obtain representations of $G$ from representations of $\mathcal U(\mathfrak g)$, this is clearly going in the wrong direction. The same happens with the second approach, as is best seen with an example: take $\A=\C[x]$, the commutative free algebra on one hermitian generator, and consider its well-behaved representation by multiplication operators on $L^2(\R)$. There is no subalgebra of $\A$ on which the corresponding norm is finite. Thus, obtaining such a simple representation by the second approach requires enlarging $\C[x]$ to contain at least one function of exponential decay---a procedure which is, again, taking as given something that, in several practical cases, should come as a result. 

Our approach is precisely based on enlarging the algebra $\A$ to contain inverses to the elements of the form $1+a^\dagger a$, thus obtaining a so-called \emph{symmetric involutive algebra.} It is known that representations of the latter are integrable~\cite{m:Inou76}. Our main result is Theorem~\ref{ppal}, which says that when the multiplicative set $\S$ generated by $\set{1+a^\dagger a | a\in\A}$ satisfies the Ore condition (see Section~\ref{ore}), integrable representations of $\A$ are in bijection with representations of the Ore localization $\A\S^{-1}$.

The paper is organized  as follows: the second section is a short reminder of necessary background, and
our results are proved in the third section.
For convenience, an elementary appendix on closable operators is included.

\section{Short reminder on representation theory}


This section contains no new results. It is included only for the convenience of non-expert readers, and it will serve to introduce our notations, too. 

\subsection{Ultracyclic representations and GNS construction} \label{alg_rep_thy}

Let $\V$ be a complex vector space and $a\in\L(\V)$. We denote their corresponding algebraic duals by
\[
\V^{\dagger} = \Set{f: \V\rightarrow\C | f \text{ is complex linear}},\quad a^{\dagger}: f\in \V^{\dagger} \mapsto fa\in \V^{\dagger}.
\]
Now, let $\A$ be a unital involutive algebra with involution $(\cdot)^\dagger$, to be represented by operators in $\L(\V)$.  Allowing for a correspondence between involution and algebraic duality requires a choice of antilinear inclusion $\V\hookrightarrow \V^{\dagger}$, which we assume given by an embedding $\V\hookrightarrow \H$ as a dense subspace in a Hilbert space.

\begin{defn}
Given a dense subspace $\V$ of a Hilbert space $\H$, denote by
\[
\LHV = \Set{ a:\V\rightarrow \V | \V\subseteq \dom{a^*},\ a^*\V\subseteq \V},
\]
which is an involutive algebra with involution $a^\dagger = a^*|_\V$.
A \emph{representation} of the involutive algebra $\A$ is a morphism $\pi:\A\rightarrow \LHV$ of involutive algebras.
We say that $\pi$ is \emph{ultracyclic} if it admits an ultracyclic vector, that is, an $\Omega\in \V$ such that $\V=\pi(\A)\Omega$.
\end{defn}


\begin{defn}
Let $\A$ be an involutive algebra and $\A_\text h$ the real vector subspace of its hermitian elements. We say that $x\in\A_\text h$ is \emph{positive} if it belongs to the cone
\[
\Pi(\A) = \Set{ \sum_{i=1}^n \lambda_ia_i^\dagger a_i | \lambda_i>0,\ a_i\in \A } \subseteq \A_\text h.
\]
\end{defn}

The cone of positive elements allows one to define an order relation on $\A$. This, in turn, gives an order relation on its algebraic dual: we say that $f\in \A^{\dagger}$ is \emph{positive} if
\[
f(x) \geq 0,\quad \forall x\in \Pi(\A).
\]
Recall that positivity implies:
\begin{enumerate}
\item
$f(a^\dagger b)=\overline{f(b^\dagger a)}$.
\item
The Cauchy-Schwartz inequality $\mod{f(a^\dagger b)}^2\leq f(a^\dagger a)f(b^\dagger b)$.
\end{enumerate}
We say that $f\in\A^{\dagger}$ is a \emph{state} if it is positive and $f(1)=1$. We denote the set of states of $\A$  by $\Sigma(\A)$.


Given a representation $\pi:\A\rightarrow \LHV$, one obtains a state $f\in\A^\dagger$ from any  vector $\Omega\in\V$ with $\norm\Omega = 1$ by the formula
\[
f(a) = \langle \Omega,\pi(a)\Omega\rangle.
\]
The GNS construction allows a recovery of both $\pi|_{\pi(\A)\Omega}$ (modulo unitary conjugation) and $\Omega$ (modulo a phase factor) from $f$, thus establishing a correspondence between ultracyclic representations and states. We proceed to describe it; for a complete treatment, see~\cite{m:Schm90}.

Let $f\in \Sigma(\A)$, and consider the set
\(
\mathcal I = \set{ a\in \A | f(a^\dagger a)=0 }.
\)
Using the Cauchy-Schwartz inequality it is easily seen that $\mathcal I$ is a left ideal. The quotient $\V = \A/\mathcal I$ is densely embedded in its Hilbert space completion $\H$ with respect to the scalar product
\[
\left\langle [a], [b]\right\rangle = f(a^\dagger b),\quad a,b\in\A,
\]
where $[\cdot]$ denotes the equivalence class in $\V$ of its argument.
The GNS representation $\pi$ is simply given by the canonical left $\A$-module structure of $\V$. It admits the ultracyclic vector $\Omega = [1]\in\V$.

\subsection{Integrability}

As mentioned in the introduction, the algebraic operations of $\LHV$ are not satisfactory from an analytic point of view. 
%
Indeed, every $a\in\LHV$ is closable because its adjoint $a^*$ is densely defined. Now, while closed operators cannot always be added or composed, the natural operations when they can are the following.

\begin{defn} \label{strong_ops}
Let $A,B\in\mathcal C(\H)$, the set of closed, densely defined operators on a Hilbert space $\H$.
If $\mathcal D = \dom A\cap \dom B$ is dense and $A|_{\mathcal D} + B|_{\mathcal D}$ is closable, we define their \emph{strong sum}
\[
A+B = \overline{A|_{\mathcal D} + B|_{\mathcal D}} \in \mathcal C(\H).
\]
Analogously, if $\mathcal D = B^{-1}\dom A$ is dense and $AB|_{\mathcal D}$ is closable, we define their \emph{strong product}
\[
AB = \overline{AB|_{\mathcal D}} \in \mathcal C(\H).
\]
\end{defn}
\begin{rk}
When operating with closed operators we will always mean strong operations. This should cause no confusions, because we will always write elements of $\mathcal C(\H)$  in uppercase, and elements of $\LHV$ in lowercase.
\end{rk}
\begin{prop} \label{A+B, AB exist}
Let $a,b\in\LHV$. If $A$ and $B$ are their respective closures, 
then  $A+B$ and  $AB$ exist.
\end{prop}
\begin{proof}
Let  $\mathcal D = \dom A\cap \dom B$. One has that
\[
\langle (A+B)\xi,\eta \rangle = \langle \xi,(a^\dagger+b^\dagger)\eta \rangle,\quad \forall \xi\in\mathcal D,\ \forall \eta\in\V,
\]
whence
\(
a^\dagger + b^\dagger \subseteq (A|_{\mathcal D} + B|_{\mathcal D})^*
\)
and $A+B$ exists. Analogously, if $\mathcal D = B^{-1}\dom A$,
\[
\langle AB\xi,\eta \rangle = \langle \xi, b^\dagger a^\dagger \eta\rangle, \quad \forall \xi\in\mathcal D,\ \forall\eta\in\V,
\]
whence $b^\dagger a^\dagger \subseteq (AB|_{\mathcal D})^*$ and $AB$ exists.
\end{proof}
\begin{rk}
The analytic procedure of closure breaks down the algebraic structure of $\LHV$. Indeed, 
\[
\overline{a+b} \subseteq \bar a+\bar b,\quad \overline{ab}\subseteq \bar a\bar b,\quad \overline{a^\dagger}\subseteq a^*,\quad \forall a,b\in\L_\H(\V),
\]
and equalities do not hold in general.
\end{rk}

Thus, given a representation $\pi:\A\rightarrow \LHV$, one should not expect that $\bar\pi$ be a representation too, where
\[
\bar\pi: a\in\A\mapsto \overline{\pi(a)} \in \mathcal C(\H).
\]
Integrable representations are, almost by definition, those for which this is actually the case.
\begin{defn}
We say that a representation $\pi:\A\rightarrow \LHV$ is \emph{integrable} if $\bar\pi(a^\dagger) = \bar\pi(a)^*$, for all $a\in\A$.
\end{defn}
\begin{prop} \label{integrable <-> closure is rep}
Let $\pi:\A\rightarrow \LHV$ be a representation.
One has that $\pi$ is integrable if, and only if, $\bar\pi(\A)\subseteq \mathcal C(\H)$ is an involutive algebra and  $\bar\pi:\A\rightarrow \bar\pi(\A)$ is a morphism.
\end{prop}
\begin{proof}
Sufficiency is obvious. For necessity, 
let $a,b\in\A$ and $A=\bar\pi(a), B=\bar\pi(b)$. From the proof of Proposition~\ref{A+B, AB exist},
\[
A+B = (A|_{\mathcal D} + B|_{\mathcal D})^{**} \subseteq \pi(a^\dagger+b^\dagger)^* = \bar\pi(a+b),
\]
where $\mathcal D= \dom A\cap \dom B$ and the last equality follows from integrability. Analogously,
\[
AB = (AB|_{\mathcal D})^{**} \subseteq \pi(b^\dagger a^\dagger)^* = \bar\pi(ab),
\]
where $\mathcal D = B^{-1}\dom A$.
\end{proof}


\subsection{Symmetric involutive algebras} \label{symm_inv_alg}

We have seen that representations which are not integrable do not really deserve their name from an analytic point of view, and that the problem originates in the fact that $\LHV$ is not a good replacement for $B(\H)$ in generalizing the theory of C*-algebras  to unbounded operator algebras.
%
Now, the need for such a generalization has long been recognized and, over time, the notion of \emph{partial involutive algebras}~\cite{m:AntoInouTrap02} has emerged as the safest candidate (because of its generality). In that approach, one actually gives up the structure of algebra, fully acknowledging the fact that strong sum and strong product are not defined all over $\mathcal C(\H)\times\mathcal C(\H)$.  Here, we will limit ourselves to work with subsets of $\mathcal C(\H)$ which are involutive algebras with respect to strong sum, strong product and operator involution. Two known examples are: the set of measurable operators affiliated with a von Neumann algebra admitting a normal, semifinite trace~\cite{m:Sega53}; and the so-called \emph{symmetric involutive algebras}~\cite{m:Inou76}, on which we base our approach to the well-behaved representations of involutive algebras.

\begin{defn}
A unital, involutive algebra is said to be \emph{symmetric} if $1+a^\dagger a$ is invertible, for all $a\in\A$. 
\end{defn}
There is a generalization to non-unital algebras, see~\cite{m:Frag05}, for instance. Representations $\pi:\A\rightarrow \LHV$ of symmetric involutive algebras enjoy several good properties, among which we mention:
\begin{itemize}
	\item They are integrable~\cite{m:Inou77}. We will revisit this in the course of this paper. 
	\item They are direct sums of cyclic representations. We remark that this is not always the case if $\A$ is any involutive algebra, see \cite[Corollary 11.6.8]{m:Schm90}.
\end{itemize}

\section{Integrable representations and Ore localization}

\subsection{Integrability and symmetry}



All over this subsection, $\pi:\A\rightarrow \LHV$ will be a representation and
\[
\S = \Set{ \prod_{i=1}^n \bigl(1+a_i^\dagger a_i\bigr) | a_i\in \A}. 
\]
If $\pi$ is integrable, it follows immediately from Propositions~\ref{integrable <-> closure is rep} and~\ref{App:fund}   that $\bar\pi(s)$ is invertible, for all $s\in\S$. Here we show that these two properties are actually equivalent, and then give an alternative proof of Inoue's result asserting that representations of symmetric involutive algebras are integrable.

\begin{lem} \label{pi(s)=S}
Let $s=\prod_{i=1}^{n}(1+a_{i}^{\dagger}a_{i}) \in \S$ and
\[
S = \prod_{i=1}^{n}(1+A_{i}^{*}A_{i}),\quad A_{i} = \bar\pi(a_{i}).
\]
If $\bar\pi(s)$ is surjective, then $\bar\pi(s) = S$. In particular, $\bar\pi(s)$ is invertible and has a bounded inverse.
\end{lem}
\begin{proof}
Observe that
\[
S = (1+A_{1}^{*}A_{1})(1+A_{2}^{*}A_{2})|_{\mathcal D_{2}}\cdots (1+A_{n}^{*}A_{n})|_{\mathcal D_{n}}
\]
where $\mathcal D_{i+1} = (1+A_{i+1}^{*}A_{i+1})^{-1}\mathcal D_{i}$ and $\mathcal D_{1} = \dom{1+A_{1}^{*}A_{1}}$. It follows that $S$ is closed (for it is invertible with bounded inverse). Since $\pi(s)\subseteq S$, we conclude that $\bar\pi(s) \subseteq S$, too. But, by hypothesis, $\bar\pi(s)$ is surjective, whence it does not admit any injective, strict extension and must be equal to $S$.
\end{proof}
\begin{coro} \label{pi(s)pi(t)=pi(st)}
If $\bar\pi(s)$ is surjective for all $s\in\S$, then
\[
\bar\pi(st) = \bar\pi(s)\bar\pi(t), \quad \forall s,t\in\S.
\]
\end{coro}
\begin{proof}
It is a straightforward consequence of Lemma \ref{pi(s)=S}. 
\end{proof}

\begin{lem} \label{pi(s)V is a core}
Let $a\in\A$ and $s\in\S$. If $\bar\pi(t)$ is surjective for all $t\in\S$, then $\bar\pi(s)\V$ is a core for both $\bar\pi(a)$ and $\bar\pi(a)^{*}$.
\end{lem}
\begin{proof}
Let $t=1+a^{\dagger}a\in\S$. From Lemma \ref{pi(s)=S} and Proposition \ref{App:fund} in the appendix, we see that
\[
\bar\pi(t)^{-1}: \left(\H,\norm{\ }\right) \rightarrow \left(\dom {\bar\pi(a)}, \norm{\ }_{\bar\pi(a)}\right)
\]
is continuous and has  dense range. Therefore, $\bar\pi(t)^{-1}\mathcal D$ will be a core for $\bar\pi(a)$, for any dense $\mathcal D\subseteq \H$. Now, consider $\mathcal D = \bar\pi(ts)\V$. It is a dense subspace, for $\bar\pi(ts)$ is surjective and $\V$ is a core for it. Using Corollary \ref{pi(s)pi(t)=pi(st)}, we conclude that
\[
\bar\pi(t)^{-1}\bar\pi(ts)\V = \bar\pi(s)\V
\]
is a core for $\bar\pi(a)$. The fact that it is also a core for $\bar\pi(a)^{*}$ follows from the same argument, applied this time to $t=1+aa^{\dagger}$. Indeed, now
\[
\bar\pi(t)^{-1}: (\H,\norm{\ }) \rightarrow \left(\dom {\bar\pi(a)^{*}}, \norm{\ }_{\bar\pi(a)^{*}}\right)
\]
is continuous and has dense range. 
\end{proof}

\begin{coro}
$\pi$ is integrable if, and only if, $\bar\pi(s)$ is surjective, for all $s\in\S$.
\end{coro}
\begin{proof}
Indeed, given $a\in\A$, $\bar\pi(a^{\dagger}) \subseteq \bar\pi(a)^{*}$. But, being $\V$ a core for both of them, they must actually coincide.
\end{proof}

%
%

We finish this subsection with the following result. As mentioned before, our main result is a partial converse.

\begin{prop} \label{S inv -> anal}
Suppose that $\A\subseteq \mathcal B$, where $\mathcal B$ is an involutive algebra such that the elements of $\S$ are invertible in it. If $\pi$ admits an extension $\tilde\pi:\mathcal B\rightarrow \LHV$, then it is integrable.
\end{prop}
\begin{proof}
Given $s = \prod_{i=1}^{n}(1+a_{i}^{\dagger}a_{i})$, we have to prove that $\bar\pi(s)$ is surjective. As in the proof of Lemma \ref{pi(s)=S},
\[
\bar\pi(s) \subseteq S = \prod_{i=1}^{n}(1+A_{i}^{*}A_{i}),\quad A_{i} = \bar\pi(a_{i}).
\]
Now, let $\xi\in\H$ and consider a sequence $\set{v_{n}}\subseteq \V$ converging to $\xi$. Since $S^{-1}$ is bounded, $u_{n} = S^{-1}v_{n}$ converges to, say, $\eta\in\H$. But, by hypothesis, $s$ is invertible in $\mathcal B$, and 
\[
u_{n} = S^{-1}v_{n} = S^{-1}\pi(s)\tilde\pi(s^{-1})v_{n} = \tilde\pi(s^{-1})v_{n},
\]
so that $u_{n}\rightarrow \eta$ and $\pi(s)u_{n} = v_n\rightarrow \xi$. By closedness, $\bar\pi(s)\eta = \xi$, as needed.
\end{proof}

\begin{rk}
As a consequence, we recover the following result due to Inoue~\cite{m:Inou76}:
if $\A$ is symmetric, then its representations are integrable.
\end{rk}

\subsection{Ore localization for involutive algebras} \label{ore}

Let $\A$ be a unital ring and $\S\subseteq \A$ a \emph{multiplicative} subset  (that is, such that $\S\S = \S$. Note that, in particular, $1\in \S$). 
Proposition \ref{S inv -> anal} leads to study the problem of adjoining inverses to the elements of $\S$.
It is not hard to see that there exists a ring $\A_\S$, called \emph{universal localization of $\A$ at $\S$,} which is a universal solution to this.
Now, when localizing, unexpected things can happen (such as ending up with a trivial ring) and it is good to have conditions controlling $\A_\S$ and, above all, enabling one to make calculations in it. 
From our point of view, having a manageable localization is important to investigate two problems:
\begin{enumerate}
	\item The extension of states from $\A$ to  $\A_\S$. By GNS construction, this amounts to obtaining integrable extensions of cyclic representations.
	\item The possibility that integrable representations actually extend to representations of $\A_\S$, providing a converse to Proposition \ref{S inv -> anal}.
\end{enumerate}
In this section we revisit the Ore construction, which deals with the particular case in which $\A_\S$ is the non-commutative analogue of a \emph{ring of fractions}. 
A complete introduction can be found in~\cite{m:GoodWarf04}.
The novelty here is that we show that the Ore construction carries over smoothly to the case of involutive algebras.

\begin{defn}
Let $\A$ be a ring.
We say that a subset $\S\subseteq \A$ is a \emph{right} (resp. \emph{left}) \emph{Ore set} if it is multiplicative and
\[
\forall (a,s) \in \A\times \S, \ \exists (b,t)\in \A\times \S, \  at = sb\ (\text{resp. } ta=bs),
\]
and we say that it is an Ore set if it is both a left and a right Ore set.
\end{defn}
\begin{rk}
If $s$ and $t$ are invertible, the equation $at = sb$ can be rewritten as $s^{-1}a = bt^{-1}$. In intuitive terms, the Ore property is establishing the possibility of writing ``non-commutative fractions'' indifferently from the right or from the left.
\end{rk}

We say that a ring morphism $f:\A\rightarrow \mathcal B$ is $\S$-inverting if $f(s)$ is invertible in $\mathcal B$, for all $s\in\S$. Then, consider the following universal problem for the $\S$-inverting morphism $\ell$:
\begin{center} \begin{tikzpicture}[description/.style={fill=white,inner sep=2pt}] 
\matrix (m) [matrix of math nodes, row sep=2.5em, 
column sep=4em, text height=1.5ex, text depth=0.25ex] 
{ \A & \A_\S \\ 
 & \mathcal B, \\ }; 
\path[->,font=\scriptsize] 
(m-1-1) edge node[auto] {$\ell$} (m-1-2) 
(m-1-1) edge node[auto] {$f$} (m-2-2)
(m-1-2) edge [dashed] node[auto] {$\tilde f$} (m-2-2);
\end{tikzpicture} \end{center} 
where $f$ is any $\S$-invertible morphism.
If $\S\subseteq \A$ is an Ore set, the solution 
can be constructed as follows.
On the set $\A\times \S$ define the equivalence relation 
\[
(a,s)\sim (b,t) \ \Leftrightarrow\ 
\exists  u,v\in \A,\ (au, su) = (bv,tv) \in \A\times \S,
\]
and denote by $[a,s]$ the equivalence class of $(a,s)$. Linear combination and multiplication of elements $[a_1,s_1], [a_2,s_2]\in (\A\times\S)/\!\!\sim$ are defined as follows:
\begin{itemize}
\item
	Choose $(b,t)\in \A\times \S$ such that $s_1t=s_2b$. Then,
	\[
		\lambda[a_1,s_1] + [a_2,s_2] = [\lambda a_1t,s_1t] + [a_2b,s_2b] = [\lambda a_1t+a_2b,s_1t].
	\]
\item
	Choose $(b,t)\in \A\times \S$ such that $a_2t=s_1b$. Then,
	\begin{align*}
	[a_1,s_1]\cdot [a_2,s_2] &= [a_1b,s_2t] \\
		\bigl(&= [a_1b,s_1b] [a_2t,s_2t].\ \bigr)
	\end{align*}
	The last equality is enclosed in parentheses because it does not necessarily make sense; we include it because it does motivate the definition.
\end{itemize}
It can be checked that this definitions  equip $\A\S^{-1} = (\A\times\S)/\!\!\sim$ with a unital ring structure, whose unit is $[1,1]$.
The morphism $\ell: \A\rightarrow \A \S^{-1}$ is defined by $a\mapsto [a,1]$. This turns out to give an inclusion if $\S$ contains uniquely \emph{regular} elements---that is, elements $s$ such that
\[
0\in \{as,sa\} \ \Rightarrow\  a=0,\quad \forall a\in \A.
\]
An analog construction to that of $\A \S^{-1}$ can be done if $\S\subseteq \A$ is a left Ore set, and in the ``bilateral'' case those two localizations coincide, modulo an isomorphism whose restriction to (the image of) $\A$ is trivial.
\begin{rk} \label{simple_prop_mult}
Before going further, we note a simple property of the multiplication. Given $(a,s)\in \A\times\S$, let $(b,t)\in \A\times\S$ be such that $at = sb$. Then, by definition,
\(
[1,s][a,1] = [b,t].
\)
Now, given any $u\in\A$ such that $us\in\S$, one still has that $uat=usb$, and therefore
\[
[1,us][ua,1] = [b,t] = [1,s][a,1].
\]
\end{rk}

We are ready to treat  the case of a unital involutive algebra $\A$, which is not dealt with in ring theory and constitutes our humble contribution to the subject. Suppose that $\S\subseteq\A$ is an Ore subset such that $\S^\dagger = \S$ (observe that, in this case, the left and right Ore conditions are equivalent). On $\A \S^{-1}$ define the following operation:
\[
[a,s]^\dagger 
	= [1,s^\dagger] [a^\dagger,1]. 
\]
By Remark \ref{simple_prop_mult}, this depends uniquely on the equivalence class of $(a,s)$ because, given $u\in\A$ such that $su\in\S$,
\[
[au,su]^\dagger  = [1,u^\dagger s^\dagger] [u^\dagger a^\dagger,1] =  [1,s^\dagger][a^\dagger,1] = [a,s]^\dagger.
\]

\begin{prop}
 $(\cdot)^\dagger: \A \S^{-1}\rightarrow \A \S^{-1}$ is an involution, and $\ell:\A\rightarrow \A \S^{-1}$ is a morphism of unital involutive algebras.
\end{prop}
\begin{proof} 
In fact, given $(a_1,s_1),(a_2,s_2)\in \A\times \S$:
\begin{itemize}
\item
	Let $(b,t)\in \A\times \S$ be such that  $s_1t=s_2b$. We have that
	\[
	\bigl( \lambda[a_1,s_1]+[a_2,s_2] \bigr)^\dagger = [\lambda a_1t+a_2b,s_1t]^\dagger = [1,t^\dagger s_1^\dagger] [\bar\lambda t^\dagger a_1^\dagger + b^\dagger a_2^\dagger,1]. 
	\]
	 On the other hand, using Remark \ref{simple_prop_mult},
	\begin{align*}
	\bar\lambda[a_1,s_1]^\dagger + [a_2,s_2]^\dagger &= [1,s_1^\dagger] [\bar\lambda a_1^\dagger,1] + [1,s_2^\dagger] [a_2^\dagger,1] \\ 
		&= [1,t^\dagger s_1^\dagger] [\bar\lambda t^\dagger a_1^\dagger,1] + [1,b^\dagger s_2^\dagger] [b^\dagger a_2^\dagger,1] \\
		&= [1,t^\dagger s_1^\dagger] [\bar\lambda t^\dagger a_1^\dagger + b^\dagger a_2^\dagger,1].
	\end{align*}
	Hence, $(\cdot)^\dagger$ is antilinear.
\item
	Let $(b,t)\in \A\times \S$ be such that  $a_2t=s_1b$. We have that
	\[
	\bigl( [a_1,s_1][a_2,s_2] \bigr)^\dagger = [a_1b,s_2t]^\dagger 
		= [1,t^\dagger s_2^\dagger][b^\dagger a_1^\dagger,1].
	\]
	On the other hand,
	\begin{align*}
	[a_2,s_2]^\dagger [a_1,s_1]^\dagger &= [1,s_2^\dagger][a_2^\dagger,1][1,s_1^\dagger][a_1^\dagger,1] = [1,s_2^\dagger][a_2^\dagger,s_1^\dagger][a_1^\dagger,1] \\
		&= [1,s_2^\dagger][1,t^\dagger][b^\dagger,1][a_1^\dagger,1] = [1,t^\dagger s_2^\dagger][b^\dagger a_1^\dagger,1].
	\end{align*}
\item
	Finally,
	\[
	[a,s]^{\dagger\dagger} = \bigl( [1,s^\dagger][a^\dagger,1] \bigr)^\dagger = [a^\dagger,1]^\dagger[1,s^\dagger]^\dagger = [a,1][1,s] 
		= [a,s]. \qedhere
	\]
\end{itemize}
\end{proof}
From now on, we will write $as^{-1}$ instead of $[a,s]$ for the elements of $\A \S^{-1}$.

\subsection{Integrable representations and representations of $\A\S^{-1}$}

All over this subsection $\A$ will be a unital involutive algebra and we will assume that
\(
\S = \Set{\prod_{i=1}^n(1+a_i^\dagger a_i) | a_i\in\A }
\)
contains only regular elements and satisfies the Ore condition.
We start by proving that every ultracyclic representation of $\A$ admits an integrable extension, which is a simple consequence of the following fact.

\begin{lem}
$\Pi(\A)$ is cofinal in $\Pi(\A \S^{-1})$.
\end{lem}
\begin{proof}
Given $(a,s)\in\A\times\S$, we will prove that $(s^{-1}a)^\dagger s^{-1}a  \leq a^\dagger a$ (recall that all the elements of $\A\S^{-1}$ can be written in the form $s^{-1}a$).
In order to do so, note the following order property of involutive algebras, which is a direct consequence of the definition: if $x\leq y$, then $a^\dagger xa\leq a^\dagger ya$. Therefore, it suffices to show that $(s^{-1})^\dagger s^{-1} \leq 1$, for all $s\in \S$---which we do next. 

Let $b\in\A$. We have that
\(
1\leq 1+ 2b^\dagger b + (b^\dagger b)^2 = (1+b^\dagger b)^2, 
\)
from which
\[
\bigl((1+b^\dagger b)^{-1}\bigr)^\dagger(1+b^\dagger b)^{-1} = (1+b^\dagger b)^{-2} \leq 1.
\]
Now, suppose that
\[
s = (1+b_1^\dagger b_1)(1+b_2^\dagger b_2)\cdots (1+b_n^\dagger b_n)\in  S,\quad b_i\in\A.
\] 
Inductively, we see that
\[
(s^{-1})^\dagger s^{-1} = (1+b_1^\dagger b_1)^{-1}\cdots (1+b_n^\dagger b_n)^{-1}(1+b_n^\dagger b_n)^{-1} \cdots (1+b_1^\dagger b_1)^{-1} \leq 1. \qedhere
\]
\end{proof}

\begin{coro}
Every ultracyclic representation $\pi:\A\rightarrow \LHV$ admits an integrable extension  $\tilde\pi:\A\rightarrow \L^\dagger(\tilde \V)$. 
\end{coro}
\begin{proof}
Let $f\in \Sigma(\A)$ be the state producing, by GNS construction, the representation $\pi$.
Since $\A_\text{h}$ is cofinal in the \emph{real} ordered vector space $(\A\S^{-1})_\text{h}$, the restriction $f_\text{h} = f|_{\A_\text{h}}$ admits a positive extension $\tilde f_\text{h}$ to all of $(\A\S^{-1})_\text{h}$  (see, for instance, \cite[Theorem 9.8]{m:Conw94}) which, in turn, uniquely determines a state  $\tilde f\in\Sigma(\A\S^{-1})$ by
\[
\tilde f(x+\I y) = \tilde f_\text{h}(x) + \I \tilde f_\text{h}(y),\quad x,y\in\A_\text{h}.
\]
By GNS construction, we get an ultracyclic representation $\tilde\pi: \A\S^{-1}\rightarrow \L^\dagger(\tilde\V)$, where 
\[
\tilde\V = (\A\S^{-1})/\set{ a\in\A\S^{-1} | \tilde f(a^\dagger a) = 0}.
\]
The inclusion $\A\subseteq \A\S^{-1}$ induces an inclusion $\V\subseteq\tilde\V$, which is well defined because 
\[
\set{a\in\A | f(a^\dagger a) = 0 } \subseteq \set{ a\in \A\S^{-1} | \tilde f(a^\dagger a) = 0}.
\]
The restriction of $\tilde\pi$ to $\A$ gives the desired extension, which is integrable thanks to Proposition~\ref{S inv -> anal}.
\end{proof}

\begin{rk} 
Results of this kind are not unknown in some important cases. For instance, in~\cite{m:BorcYngv90} Borchers and Yngvason work out the case of Borchers algebras, appearing naturally in quantum field theory.
\end{rk}

Next we prove our main result. We start with the following crucial lemma.

\begin{lem} \label{as^{-1}=t^{-1}b}
Suppose that $\pi$ is integrable, and let $(a,s),(b,t)\in\A\times\S$ be such that $ta=bs$. Then,
\[
\bar\pi(a)\bar\pi(s)^{-1} \supseteq \bar\pi(t)^{-1}\bar\pi(b).
\] 
In particular, 
\(
\bar\pi(s)^{-1}\V\subseteq\dom{\bar\pi(a)},
\)
for all $(a,s)\in \A\times\S$.
\end{lem}
\begin{proof}
Indeed, starting from $\pi(t)\pi(a) = \pi(b)\pi(s)$ we find, pre-multiplying by $\bar\pi(t)^{-1}$ and post-multiplying by $\bar\pi(s)^{-1}|_{\pi(s)\V}$, that
\[
\bar\pi(a)\bar\pi(s)^{-1}|_{\pi(s)\V} = \bar\pi(t)^{-1}\bar\pi(b)|_{\pi(s)\V}.
\]
Now, given $\xi\in\dom{\bar\pi(b)}$, by Lemma \ref{pi(s)V is a core} there exists a sequence $[sc_n] \rightarrow\xi$ such that $[bsc_n] \rightarrow \bar\pi(b)\xi$. Since $\bar\pi(t)^{-1}$ is bounded, we see that 
\[
\bar\pi(a)[c_n] = \bar\pi(a)\bar\pi(s)^{-1}[sc_n] =  \bar\pi(t)^{-1}[bsc_n]
\]
converges. Therefore, $\bar\pi(s)^{-1}\xi\in\dom{\bar\pi(a)}$ and $\bar\pi(a)\bar\pi(s)^{-1}\xi = \bar\pi(t)^{-1}\bar\pi(b)\xi$, as needed.
\end{proof}

\begin{thm} \label{ppal}
There exists a bijective correspondence between integrable representations of $\A$ and representations of $\A\S^{-1}$. 
\end{thm}
\begin{proof}
By Proposition \ref{S inv -> anal}, representations of $\A\S^{-1}$ induce, by restriction, integrable representations of $\A$. In the other direction,
let $\pi:\A\rightarrow \LHV$ be an integrable representation. We extend it to $\A\S^{-1}$ as follows. Define
\[
\tilde \V = \text{span }\biggl( \bigcup_{s\in \mathcal S} \bar\pi(s)^{-1}\V \biggr) \supseteq \V.
\]
Note that, by Lemma \ref{as^{-1}=t^{-1}b}, $\tilde\V\subseteq \dom{\bar\pi(a)}$, for all $a\in\A$. We will show that $\tilde\pi(a) := \bar\pi(a)|_{\tilde\V}$ belongs to $\L^\dagger(\tilde \V)$, which means verifying that:
\begin{enumerate}
	\item
$\bar\pi(a)\tilde \V\subseteq \tilde \V$. Indeed, given $s\in\S$, let $(b,t)\in\A\times\S$ be such that $ta=bs$. One has that $\bar\pi(a)\bar\pi(s)^{-1}\V = \bar\pi(t)^{-1}\bar\pi(b)\V \subseteq \bar\pi(t)^{-1}\V \subseteq \tilde\V$.
	\item
$\tilde \V\subseteq \dom {(\bar\pi(a)|_{\tilde \V})^*} = \dom {\bar\pi(a)^*}$. By integrability, $\bar\pi(a)^* = \bar\pi(a^\dagger)$ and we already know that $\tilde \V\subseteq \dom{\bar\pi(a^\dagger)}$.
	\item
$(\bar\pi(a)|_{\tilde \V})^*\tilde \V\subseteq\tilde \V$. Again, follows by integrability and the already proved fact that $\bar\pi(a^\dagger)\tilde \V\subseteq \tilde \V$.
\end{enumerate}
%
By the universal property of the Ore localization, the morphism $\tilde\pi$ will factor through $\A\S^{-1}$ if we prove, given $s\in \S$, that $\tilde\pi(s)$ is invertible in $\L^\dagger(\tilde \V)$. This means verifying that $\bar\pi(s)^{-1}|_{\tilde \V}\in\L^\dagger(\tilde \V)$. But, by integrability, $\tilde \V$ is invariant under $\bar\pi(s)^{-1}$:
\[ 
\bar\pi(s)^{-1}\tilde \V = \bigcup_{t\in S} \bar\pi(s)^{-1}\bar\pi(t)^{-1}\V = \bigcup_{t\in S} \bar\pi(ts)^{-1}\V 
	\subseteq \tilde \V.
\] 
The same calculation shows that $\tilde \V$ is invariant under
\(
( \bar\pi(s)^{-1} )^* = \bar\pi(s^\dagger)^{-1},
\)
and the conclusion follows.
\end{proof}

\begin{rk}
An important question arises: do integrable representations always admit an extension to  the universal localization $\A_\S$? Our proof makes essential use of the Ore condition, which conceivable might not always hold for subsets of $\mathcal C(\H)$ which are algebras.
\end{rk}

\section*{Appendix: short reminder on unbounded operators}

Let $\H$ be a Hilbert space and $a:\dom a\subseteq\H \rightarrow \H$ a linear operator, not necessarily bounded. We say that $a$ is \emph{closed} if its graph
\[
\gr(a) = \set{\xi\oplus a\xi | \xi\in\H } \subseteq \H\oplus\H
\]
is closed, and that $a$ is \emph{closable} if $\overline{\gr(a)}$ is a graph, namely if
\[
0\oplus\eta\in \overline{\gr(a)} \Rightarrow \eta=0,\quad \forall\eta\in\H.
\]
In this case, there exists a unique operator $\bar a$ such that $\overline{\gr(a)} = \gr(\bar a)$, which is called the \emph{closure} of $a$. 
\begin{rk}
If a closed operator was defined all over $\H$, by the closed graph theorem it would be bounded. While this is not always the case, it is expected that closed operators inherit some of the good behaviour of bounded operators.
\end{rk}

In order to define the adjoint $a^*$ of the unbounded operator $a$, we consider the linear function
\[
f_\xi = \langle \xi,a(\cdot) \rangle: \dom a\rightarrow \C.
\]
If $f_\xi$ had an extension $\bar f_\xi\in\H^*$, then $a^*\xi$ should be its inverse image via the antilinear isomorphism $\iota:\H\rightarrow\H^*$ induced by the inner product. Define, whence, 
\[
\dom {a^*} = \Set{\xi\in\H | \exists C(\xi)>0, \forall\eta\in\dom a,\ \mod{ \langle \xi,a\eta \rangle } \leq C(\xi)\|\eta\| }.
\]
By Riesz's theorem, if $\xi\in\dom {a^*}$, then $a^*\xi = \iota^{-1}(\bar f_\xi)$ exists. 
\begin{rk}
$\dom {a^*}$ has no reason to be dense in $\H$.
\end{rk}
A geometrical interpretation for the adjoint operator can be obtained through the isometry
\[
J: \xi\oplus\eta\in \H\oplus\H \mapsto (-\eta)\oplus\xi \in\H\oplus\H.
\]
Indeed, it is easily shown that
\(
\gr(a^*) = \bigl(J\gr(a)\bigr)^\perp.
\)
In particular, $a^*$ is closed. Another consequence is the following proposition.
\begin{prop}
$\dom {a^*}$ is dense in $\H$ if, and only if, $a$ is closable. Moreover, in this case $a^{**} = \bar a$.
\end{prop}

An operator with dense domain $a:\dom a\subseteq\H\rightarrow \H$ is \emph{hermitian} if
\[
\langle \xi,a\eta \rangle = \langle a\xi,\eta \rangle,\quad \forall \xi,\eta\in\dom a.
\]
This means that $a^*$ is an extension of $a$, which we denote by $a\subseteq a^*$. In particular, $\dom a\subseteq \dom {a^*}$ and therefore $a$ is closable. Since the closure of $a$ is a minimal closed extension, we have that
\[
a\subseteq \bar a \subseteq a^*.
\]
If $\bar a = a^*$, we say that $a$ is \emph{essentially self-adjoint} and that $\bar a$ is \emph{self-adjoint}---but, of course, this is not always the case.

Now we consider closed operators. In order to distinguish them from arbitrary unbounded operators, we write them in capital letters.
\begin{defn}
A subspace $\V\subseteq \H$ is said to be a \emph{core} for the closed operator $A:\dom A\subseteq \H\rightarrow \H$ if $\V\subseteq \dom A$ and $\overline{A|_\V} = A$.
\end{defn}
By definition, a core $\V$ for $A$ is such that $\set{ \xi\oplus A\xi | \xi\in \V }$ is dense in $\gr(A)$. Now, the isomorphism of vector spaces
\[
\xi\in\dom A \mapsto \xi\oplus A\xi \in \gr(A)
\]
enables one to pass the inner product of $\H\oplus\H$ to $\dom A$: we obtain
\[
\langle \xi,\eta \rangle_A = \langle \xi,\eta \rangle + \langle A\xi,A\eta \rangle.
\]
In terms of this space, $\V$ is a core for $A$ if, and only if, it is dense in $(\dom A,\norm{\ }_A)$.
We conclude the appendix with the statement of an important result---see~\cite{m:Conw94}, for example, for a proof.

\begin{prop} \label{App:fund}
Let $A:\dom A\subseteq\H\rightarrow\H$ be a closed operator. One has that:
\begin{enumerate}
\item
	$1+A^*A:A^{-1}\dom {A^*}\rightarrow \H$ is a bijection, and its inverse
	\[
	(1+A^*A)^{-1}:\H\rightarrow A^{-1}\dom {A^*}\subseteq \dom A 
	\]
	is bounded as an operator $(\H,\norm{\ }) \rightarrow (\dom A,\norm{\ }_A)$.
\item
	$\dom {A^*A} = A^{-1}\dom {A^*}$ is a core for $A$.
\end{enumerate}
\end{prop}

\bibliographystyle{amsplain}
\bibliography{math,phys,mphy}

\end{document}